\documentclass{amsart}
\usepackage{graphicx}
\usepackage{hyperref}
\usepackage{amssymb}
\usepackage{amsthm}
\usepackage{amsmath}

\newtheorem{theorem}{Theorem}[section]

\newtheorem{corollary}[theorem]{Corollary}
\newtheorem{proposition}[theorem]{Proposition}
\theoremstyle{definition}

\newtheorem{example}[theorem]{Example}

\theoremstyle{remark}
\newtheorem{remark}[theorem]{Remark}

\numberwithin{equation}{section}

\def\func#1{\mathop{\rm #1}\nolimits}

\begin{document}
\title{Neighborhoods of univalent functions}
\author{Mihai N. Pascu}
\address{Faculty of Mathematics and Computer Science, Transilvania University of Bra\c{s}ov, Bra\c{s}ov -- 500091, ROMANIA}
\email{mihai.pascu@unitbv.ro}
\thanks{The authors kindly acknowledge the support from CNCSIS research grant PNII - IDEI 209.}

\author{Nicolae R. Pascu}
\address{Department of Mathematics, Southern Polytechnic State University,
1100 S. Marietta Pkwy, Marietta, GA 30060-2896, USA.}
\email{npascu@spsu.edu}

\maketitle

\begin{abstract}
The main result shows a small perturbation of a univalent function
is again a univalent function, hence a univalent function has a
neighborhood consisting entirely of univalent functions.

For the particular choice of a linear function in the hypothesis of
the main theorem, we obtain a corollary which is equivalent to the
classical Noshiro-Warschawski-Wolff univalence criterion.

We also present an application of the main result in terms of Taylor
series, and we show that the hypothesis of our main result is sharp.
\end{abstract}

\section{Introduction}

We denote by $U_{r}=\left\{ z\in \mathbb{C}:\left\vert z\right\vert
<r\right\} $ the open disk of radius $r>0$ centered at the origin and we let
$U=U_{1}$. The class of functions $f:D\rightarrow \mathbb{C}$ analytic in
the domain $D$ will be denoted by $\mathcal{A}\left( D\right) $.

It is known that if $f:D\rightarrow \mathbb{C}$ is a univalent map in a
domain $D$, then $f^{\prime }\neq 0$ in $D$. The non-vanishing of the
derivative of an analytic function (local univalence) is not in general
sufficient to insure the univalence of the function, as it can be seen by
considering for example the exponential function $f\left( z\right) =e^{z}$
defined in the upper half-plane.

The classical Noshiro-Warschawski-Wolff univalence criterion gives a partial
converse of the above result, as follows:

\begin{theorem}
If $f:D\rightarrow \mathbb{C}$ is analytic in the convex domain $D$
and
\begin{equation*}
\func{Re}f^{\prime }\left( z\right) >0,\qquad z\in D,
\end{equation*}%
then $f$ is univalent in $D$.
\end{theorem}

In the present paper we introduce the constant $K\left( f,D\right) $
associated with a function $f:D\rightarrow \mathbb{C}$ analytic in a
domain $D$, which is a measure of the "degree of univalence" of $f$
(see Proposition \ref{characterization of univalence} and the remark
following it).

Using the constant $K\left( f,D\right) $ thus introduced, in Theorem
\ref{main theorem} we obtain a sufficient condition for univalence,
which shows that a small perturbation of a univalent function is
again univalent. As a theoretical consequence of this result, it
follows that a univalent function has a neighborhood consisting
entirely of univalent functions (see Remark \ref{Nbds of univalent
functions}).

The Theorem \ref{main theorem} is sharp, in the sense that we cannot
replace the upper bound appearing in the hypothesis of this theorem
by a larger one, as shown in Example \ref{Exemplul 1}.

For the particular choice of a linear function in Theorem \ref{main
theorem}, we obtain a simple sufficient condition for univalence
(Corollary \ref{corollary of main theorem 2}), which is shown to be
equivalent to the Noshiro-Warschawski-Wolff univalence criterion.
The main result in Theorem \ref{main theorem} can be viewed
therefore as a generalization of this classical result, in which the
linear function is replaced by a general univalent function.

The paper concludes with another application of the main result in
the case of analytic functions defined in the unit disk. Thus, in
Theorem \ref{Application to Taylor series} and the corollary
following it, we obtain sufficient conditions for the univalence of
an analytic function defined in the unit disk in terms of the
coefficients of its Taylor series representation, which might be of
independent interest.

\section{Main results}

Given a function $f:D\rightarrow \mathbb{C}$ analytic in the domain $D$ we
introduce the constant $K\left( f,D\right) $ defined as follows:

\begin{equation}
K\left( f,D\right) =\inf_{\substack{ a,b\in D  \\ a\neq b}}\left\vert \frac{%
f\left( a\right) -f\left( b\right) }{a-b}\right\vert
\end{equation}

Note that from the definition follows immediately that if the function $f$
is not univalent in $D$ then $K\left( f,D\right) =0$ The constant $K\left(
f,D\right) $ characterizes the univalence of the function $f$ in $D$  in the
following sense:

\begin{proposition}
\label{characterization of univalence}Let $f:D\rightarrow
\mathbb{C}$ be an analytic function in the domain $D$. If $K\left(
f,D\right) >0$ then $f$ is univalent in $D$.

Conversely, if $f$ is univalent in $D$ and $\Omega \subset \overline{\Omega }%
\subset D$ is a domain strictly contained in $D$, then $K\left( f,\Omega
\right) >0$.
\end{proposition}

\begin{proof}
The first statement follows from the inequality
\begin{equation*}
\left\vert f\left( a\right) -f\left( b\right) \right\vert \geq \left\vert
a-b\right\vert K\left( f,D\right) >0,
\end{equation*}%
for any distinct points $a,b\in D$.

To prove the converse, note that%
\begin{equation*}
K\left( f,\Omega \right) \geq \inf_{a,b\in \overline{\Omega }}\left\vert
F\left( a,b\right) \right\vert ,
\end{equation*}%
where $F:\overline{\Omega }\times \overline{\Omega }\rightarrow \mathbb{C}$
is the function defined by%
\begin{equation*}
F\left( a,b\right) =\left\{
\begin{array}{l}
\frac{f\left( a\right) -f\left( b\right) }{a-b},\qquad a\neq b \\
f^{\prime }\left( a\right) ,\qquad \quad a=b%
\end{array}%
\right. .
\end{equation*}

Note that since $f:D\rightarrow \mathbb{C}$ is analytic, $F$ is continuous
on the closed set $\overline{\Omega }\times \overline{\Omega }$, and
therefore $F$ attains its minimum modulus on this set:%
\begin{equation*}
K\left( f,\Omega \right) \geq \inf_{a,b\in \overline{\Omega }}\left\vert
F\left( a,b\right) \right\vert =\left\vert F\left( \alpha ,\beta \right)
\right\vert ,
\end{equation*}%
for some $\alpha ,\beta \in \overline{\Omega }$.

If $\alpha \neq \beta $, then $\left\vert F\left( \alpha ,\beta \right)
\right\vert =\left\vert \frac{f\left( \alpha \right) -f\left( \beta \right)
}{\alpha -\beta }\right\vert >0$ since $\alpha ,\beta \in \overline{\Omega }%
\subset D$ and $f$ is univalent in $D$, and if $\alpha =\beta $ then $%
\left\vert F\left( \alpha ,\alpha \right) \right\vert =\left\vert f^{\prime
}\left( \alpha \right) \right\vert >0$, again by the univalence of $f$ in $D$%
. It follows that in all cases we have%
\begin{equation*}
K\left( f,\Omega \right) \geq \left\vert F\left( \alpha ,\beta \right)
\right\vert >0,
\end{equation*}%
concluding the proof.
\end{proof}

\begin{remark}
\label{K might be 0}Note that the converse in the above proposition may not
hold for $\Omega =D$ without the additional hypothesis, as shown in the
example below.

In order to have the equivalence%
\begin{equation*}
f\text{ univalent in }D\Longleftrightarrow \text{ }K\left( f,D\right) >0,
\end{equation*}%
one needs additional hypotheses, which guarantee the existence of a
continuous extension of $f,f^{\prime }$ to $\overline{D}$, such that
$f$ is injective on $\overline{D}$ and $f^{\prime }\neq 0$ in
$\overline{D}$.

For example, in the case $D=U$, if the boundary of the image domain $f\left(
U\right) $ is a Jordan curve of class $C^{1,\alpha }$ $\left( 0<\alpha
<1\right) $, by Carath\'{e}odory theorem the function $f$ has a continuous
injective extension to $\overline{D}$, and also, by Kelogg-Warschawski
theorem, the function $f^{\prime }$ has continuous extension to $\overline{D}
$, with $f^{\prime }\neq 0$ in $\overline{D}$ (see for example \cite%
{Pommerenke}, p. 24 and pp. 48 -- 49). Following the proof above with $%
\Omega $ replaced by $U$, we obtain $K\left( f,U\right) >0$, and therefore
in this case we have
\begin{equation*}
f\text{ univalent in }U\Longleftrightarrow \text{ }K\left( f,U\right) >0.
\end{equation*}
\end{remark}

\begin{example}
Let $D=U-\left[ 0,1\right] $ be the unit disk with a slit along the positive
real axis. Since $D$ is simply connected, there exists a conformal map $%
f:U\rightarrow D$ between the unit disk $U$ and $D$ (see Figure \ref{Figura
1} below). The map $f$ has a continuous extension to $\overline{U}$, and
without loss of generality we may assume that there exists $\theta \in
(0,2\pi )$ such that $f\left( e^{i\theta }\right) =f\left( e^{-i\theta
}\right) \in \left( 0,1\right) $.

The function $f$ is univalent in $U$, but $K\left( f,U\right) =0$ since
\begin{equation*}
K\left( f,U\right) \leq \lim_{\substack{ a\rightarrow e^{i\theta }  \\ %
b\rightarrow e^{-i\theta }}}\left\vert \frac{f\left( a\right) -f\left(
b\right) }{a-b}\right\vert =\left\vert \frac{f\left( e^{i\theta }\right)
-f\left( e^{-i\theta }\right) }{e^{i\theta }-e^{-i\theta }}\right\vert =0.
\end{equation*}
\end{example}

\begin{figure}[thb]
\begin{center}
\includegraphics[scale=0.6]{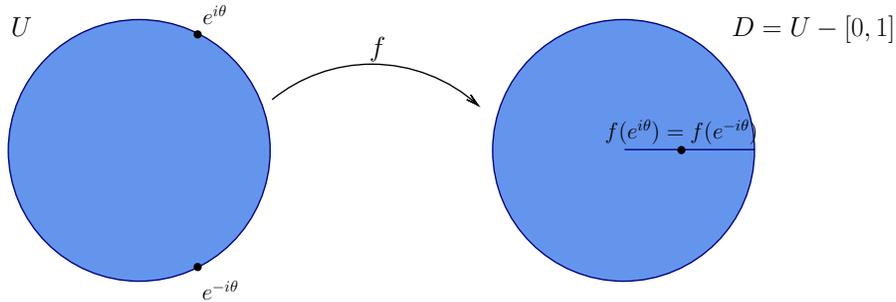}
\caption{An example of a univalent function in $U$ for which
$K\left( f,U\right) =0$.}
\label{Figura 1}
\end{center}
\end{figure}

The main result is contained in the following:

\begin{theorem}
\label{main theorem}Let $f:D\rightarrow \mathbb{C}$ be a non-constant
analytic function in the convex domain $D$. If there exists an analytic
function $g:D\rightarrow \mathbb{C}$ univalent in $D$ such that%
\begin{equation}
\left\vert f^{\prime }\left( z\right) -g^{\prime }\left( z\right)
\right\vert \leq K\left( g,D\right) ,\qquad z\in D,
\label{sufficient condition for univalency}
\end{equation}%
then the function $f$ is also univalent in $D$.
\end{theorem}

\begin{proof}
Assuming that $f$ is not univalent in $D$, there exists distinct points $%
z_{1,2}\in D$ such that $f\left( z_{1}\right) =f\left( z_{2}\right) $.
Integrating the derivative of $f-g$ along the line segment $\left[
z_{1},z_{2}\right] \subset D$ and using the hypothesis (\ref{sufficient
condition for univalency}) we obtain%
\begin{eqnarray*}
\left\vert g\left( z_{2}\right) -g\left( z_{1}\right) \right\vert
&=&\left\vert \left( f\left( z_{2}\right) -g\left( z_{2}\right) \right)
-\left( f\left( z_{1}\right) -g\left( z_{1}\right) \right) \right\vert  \\
&=&\left\vert \int_{\left[ z_{1},z_{2}\right] }f^{\prime }\left( z\right)
-g^{\prime }\left( z\right) dz\right\vert  \\
&\leq &\int_{\left[ z_{1},z_{2}\right] }\left\vert f^{\prime }\left(
z\right) -g^{\prime }\left( z\right) \right\vert \left\vert dz\right\vert  \\
&\leq &\int_{\left[ z_{1},z_{2}\right] }K\left( g,D\right) \left\vert
dz\right\vert  \\
&=&K\left( g,D\right) \left\vert z_{1}-z_{2}\right\vert .
\end{eqnarray*}

Since the points $z_{1,2}$ are assumed to be distinct, from the definition
of the constant $K\left( g,D\right) $ we obtain equivalently%
\begin{equation}
\left\vert \frac{g\left( z_{2}\right) -g\left( z_{1}\right) }{z_{2}-z_{1}}%
\right\vert \leq K\left( g,D\right) =\inf_{\substack{ a,b\in D \\ a\neq b}}%
\left\vert \frac{g\left( a\right) -g\left( b\right) }{a-b}\right\vert \leq
\left\vert \frac{g\left( z_{2}\right) -g\left( z_{1}\right) }{z_{2}-z_{1}}%
\right\vert ,
\end{equation}%
and therefore%
\begin{equation}
K\left( g,D\right) =\inf_{\substack{ a,b\in D \\ a\neq b}}\left\vert \frac{%
g\left( a\right) -g\left( b\right) }{a-b}\right\vert =\left\vert \frac{%
g\left( z_{2}\right) -g\left( z_{1}\right) }{z_{2}-z_{1}}\right\vert .
\label{minimum attained}
\end{equation}

Consider now the auxiliary function $G:D-\left\{ z_{2}\right\} \rightarrow
\mathbb{C}$ defined by
\begin{equation}
G\left( z\right) =\frac{g\left( z\right) -g\left( z_{2}\right) }{z-z_{2}}%
,\qquad z\in D-\left\{ z_{2}\right\} ,
\end{equation}%
and note that since $g$ is analytic in $D$, $G$ is also analytic in $%
D-\left\{ z_{2}\right\} $ and moreover the limit
\begin{equation}
\lim_{z\rightarrow z_{2}}G\left( z\right) =\lim_{z\rightarrow z_{2}}\frac{%
g\left( z\right) -g\left( z_{2}\right) }{z-z_{2}}=g^{\prime }\left(
z_{2}\right)
\end{equation}%
exists and it is finite. The function $G$ can be therefore extended by
continuity to an analytic function in $D$, denoted also by $G$.

Since%
\begin{equation*}
\inf_{z\in D}\left\vert G\left( z\right) \right\vert =\inf_{\substack{ z\in D
\\ z\neq z_{2}}}\left\vert G\left( z\right) \right\vert =\inf_{\substack{ %
z\in D \\ z\neq z_{2}}}\left\vert \frac{g\left( z\right) -g\left(
z_{2}\right) }{z-z_{2}}\right\vert \geq \inf_{\substack{ a,b\in D \\
a\neq b }}\left\vert \frac{g\left( a\right) -g\left( b\right)
}{a-b}\right\vert =K\left( g,D\right) ,
\end{equation*}%
combining with (\ref{minimum attained}) we obtain that
\begin{equation*}
\inf_{z\in D}\left\vert G\left( z\right) \right\vert \geq K\left( g,D\right)
=\left\vert \frac{g\left( z_{2}\right) -g\left( z_{1}\right) }{z_{2}-z_{1}}%
\right\vert =\left\vert G\left( z_{1}\right) \right\vert \geq \inf_{z\in
D}\left\vert G\left( z\right) \right\vert ,
\end{equation*}%
which shows that minimum value of the modulus of $G$ in $D$ is attained at $%
z_{1}$:%
\begin{equation*}
\inf_{z\in D}\left\vert G\left( z\right) \right\vert =\left\vert G\left(
z_{1}\right) \right\vert .
\end{equation*}

However, since the function $g$ is univalent in $D$, from the definition of $%
G$ it follows that $G\left( z\right) \neq 0$ for any $z\in D-\left\{
z_{2}\right\} $, and also $G\left( z_{2}\right) =g^{\prime }\left(
z_{2}\right) \neq 0$, and therefore the function $G$ does not vanish in $D$.
Applying the maximum modulus principle to the analytic function $1/G$ it
follows that $\left\vert G\right\vert $ must be constant in $D$, and
therefore $G$ is constant in $D$.

It follows that
\begin{equation}
g\left( z\right) =g\left( z_{2}\right) +c\left( z-z_{2}\right) ,\qquad z\in
D,  \label{g must be linear}
\end{equation}%
for a certain constant $c\in \mathbb{C}$ (from the definition of $G$ it can
be seen that the constant $c$ can be written in the form $c=g^{\prime
}\left( z_{2}\right) e^{i\theta }$, for some $\theta \in \mathbb{R}$).

The relation (\ref{g must be linear}) shows that $g$ is a linear function,
and therefore the constant $K\left( g,D\right) $ becomes in this case%
\begin{eqnarray*}
K\left( g,D\right)  &=&\inf_{\substack{ a,b\in D \\ a\neq b}}\left\vert
\frac{g\left( a\right) -g\left( b\right) }{a-b}\right\vert  \\
&=&\inf_{\substack{ a,b\in D \\ a\neq b}}\left\vert \frac{\left( g\left(
z_{2}\right) +c\left( a-z_{2}\right) \right) -\left( g\left( z_{2}\right)
+c\left( b-z_{2}\right) \right) }{a-b}\right\vert  \\
&=&\inf_{\substack{ a,b\in D \\ a\neq b}}\left\vert \frac{c\left( a-b\right)
}{a-b}\right\vert  \\
&=&\left\vert c\right\vert .
\end{eqnarray*}

The hypothesis (\ref{sufficient condition for univalency}) of the theorem
can be written therefore as follows%
\begin{equation*}
\left\vert f^{\prime }\left( z\right) -c\right\vert \leq \left\vert
c\right\vert ,\qquad z\in D,
\end{equation*}%
which shows that either $f$ is linear in $D$ (and thus univalent, since $f$
is assumed to be non-constant in $D$), or the following strict inequality
holds%
\begin{equation*}
\left\vert f^{\prime }\left( z\right) -c\right\vert <\left\vert c\right\vert
,\qquad z\in D.
\end{equation*}

Repeating the proof above with $g\left( z\right) \equiv cz$ we obtain%
\begin{eqnarray*}
\left\vert cz_{2}-cz_{1}\right\vert &=&\left\vert \left( f\left(
z_{2}\right) -cz_{2}\right) -\left( f\left( z_{1}\right)
-cz_{1}\right) \right\vert \\
&=&\left\vert \int_{\left[ z_{1},z_{2}\right] }f^{\prime }\left(
z\right) -cdz\right\vert \\
& \leq & \int_{\left[ z_{1},z_{2}\right] }\left\vert f^{\prime
}\left( z\right) -c\right\vert \left\vert dz\right\vert
\\
&<& \left\vert c\right\vert \left\vert z_{2}-z_{1}\right\vert ,
\end{eqnarray*}%
a contradiction.

The contradiction obtained shows that the function $f$ is univalent in $D$,
concluding the proof of the theorem.
\end{proof}

In the particular case $D=U$, from the previous thereom we obtain
immediately the following sufficient criterion for univalence in the unit
disk:

\begin{theorem}
\label{main theorem 2}Let $f:U\rightarrow \mathbb{C}$ be a non-constant
analytic function in the unit disk. If there exists an analytic function $%
g:U\rightarrow \mathbb{C}$ univalent in $U$ such that%
\begin{equation}
\left\vert f^{\prime }\left( z\right) -g^{\prime }\left( z\right)
\right\vert \leq K\left( g,U\right) ,\qquad z\in U,
\label{sufficient condition for univalency 2}
\end{equation}%
then the function $f$ is also univalent in $U$.
\end{theorem}

As a corollary of Theorem \ref{main theorem} we obtain immediately
the following:

\begin{corollary}
\label{corollary of main theorem 2}If $f:D\rightarrow \mathbb{C}$ is
non-constant and analytic in the convex domain $D$ and there exists $c>0$ such that%
\begin{equation}
\left\vert f^{\prime }\left( z\right) -c\right\vert \leq c,\qquad
z\in D, \label{Noshiro-Warschawski-Wolff type condition}
\end{equation}%
then $f$ is univalent in $D$.
\end{corollary}

\begin{proof}
Considering the univalent function $g:D\rightarrow \mathbb{C}$
defined by $g\left( z\right) =cz$, we have $g^{\prime }\left(
z\right) =c$
for $z\in D$ and%
\begin{equation*}
K\left( g,D\right) =\inf_{\substack{ a,b\in D \\ a\neq b}}\left\vert
\frac{g\left( a\right) -g\left( b\right) }{a-b}\right\vert =\inf_{\substack{ %
a,b\in D \\ a\neq b}}\left\vert \frac{ca-cb}{a-b}\right\vert =c,
\end{equation*}%
and therefore the claim follows from Theorem \ref{main theorem}
above.
\end{proof}

\begin{remark}
Let us note that the previous corollary can also be obtained as a direct
consequence of the classical Noshiro-Warschawski-Wolff univalence criterion,
since the hypothesis (\ref{Noshiro-Warschawski-Wolff type condition})
implies the hypothesis%
\begin{equation}
\func{Re}f^{\prime }\left( z\right) >0,\qquad z\in D.
\label{Noshiro-Warschawski-Wolff hypothesis}
\end{equation}%
of this theorem (the fact that the above inequality is a strict
inequality follows from the maximum principle, the function $f$
being assumed to be non-constant in $D$).

Conversely, the Noshiro-Warschawski-Wolff univalence criterion
follows from the previous corollary. To see this, note that in order
to prove the univalence of $f$, it suffices to prove the univalence
of $f$ in $D_{r}=r D$, for an arbitrarily fixed $r\in (0,1)$.

If the condition (\ref{Noshiro-Warschawski-Wolff hypothesis}) holds,
there exists $c>0$ such that $$f^{\prime }\left( D_{r}\right)
\subset \left\{ w\in
\mathbb{C}:\left\vert w-c\right\vert <c\right\} ,$$ or equivalent%
\begin{equation*}
\left\vert f^{\prime }\left( z\right) -c \right\vert <c,\qquad z\in
D_{r}.
\end{equation*}%

Applying Corollary \ref{corollary of main theorem 2} to the
restriction of of $f$ to $D_r$, it follows that the function $f$ is
univalent in $D_{r}.$ Since $r\in \left( 0,1\right) $ was
arbitrarily fixed, it follows that $f$ is univalent in $U$,
concluding the proof of the claim.
\end{remark}

The remark above shows that Corollary \ref{corollary of main theorem
2} and the Noshiro-Warschawski-Wolff univalence criterion are
equivalent, and therefore Theorem \ref{main theorem} is a
generalization of it. The Noshiro-Warschawski-Wolff univalence
criterion can be viewed as a particular case of the main Theorem
\ref{main theorem}, corresponding to the choice of a linear function
$g$.

\begin{remark}
\label{Nbds of univalent functions}Fixing an arbitrarily univalent function $%
g:U\rightarrow \mathbb{C}$ for which $K\left( g,U\right) \neq 0$ (see Remark %
\ref{K might be 0} above), Theorem \ref{main theorem 2} shows that a whole
neighborhood $V\left( g\right) =\left\{ f\in \mathcal{A}:\left\vert
\left\vert f^{\prime }-g^{\prime }\right\vert \right\vert \leq K\left(
g,U\right) \right\} $ of $g$ consists entirely of univalent functions in $U$
($\left\vert \left\vert \cdot \right\vert \right\vert $ denotes here the
supremum norm in the space $\mathcal{A}_{0}=\left\{ f\in \mathcal{A}:f\left(
0\right) =0\right\} $ of normalized analytic functions). Loosely stated,
Theorem \ref{main theorem 2} shows that an univalent function has a
neighborhood consisting entirely of univalent functions.
\end{remark}

The hypotheses of Theorem \ref{main theorem} and Theorem \ref{main theorem 2}
are sharp, in the sense that we cannot replace the right side of the
inequalities (\ref{sufficient condition for univalency}), respectively (\ref%
{sufficient condition for univalency 2}), by larger constants, as can be
seen from the following example.

\begin{example}
\label{Exemplul 1}Consider the function $f:U\rightarrow \mathbb{C}$ defined
by $f\left( z\right) =z+az^{2}$, $z\in U$, where $a\in \mathbb{C}$ is a
parameter.

Using Theorem \ref{main theorem 2} above with $g\left( z\right) \equiv z$,
for which $K\left( g,U\right) =1$, we obtain that the function $f$ is
univalent in $U$ if
\begin{equation*}
\left\vert 2az\right\vert \le 1,\qquad z\in U,
\end{equation*}%
that is if $\left\vert 2a\right\vert \leq 1$.

This result is sharp, since the function $f$ is univalent iff $\left\vert
a\right\vert \leq \frac{1}{2}$, as it can be checked by direct computation.
\end{example}

The univalence of the function $f$ in the previous example can also be
obtained by using the Noshiro-Warschawski-Wolff univalence criterion (for $%
\left\vert a\right\vert \leq 1/2$ we have $\func{Re}f^{\prime }\left(
z\right) >0$ for any $z\in U$). The next example shows that we may still use
Theorem \ref{main theorem 2} also in situations when the
Noshiro-Warschawski-Wolff univalence criterion cannot be applied:

\begin{example}
\label{Exemplul 2}Consider the linear map $g:U\rightarrow \mathbb{C}$
defined by $g\left( z\right) =\frac{z}{1-z}$. The function $g$ is univalent
in $U$ and we have%
\begin{equation*}
K\left( g,U\right) =\inf_{\substack{ a,b\in U \\ a\neq b}}\left\vert \frac{%
g\left( a\right) -g\left( b\right) }{a-b}\right\vert =\inf_{\substack{ %
a,b\in U \\ a\neq b}}\left\vert \frac{\frac{a}{1-a}-\frac{b}{1-b}}{a-b}%
\right\vert =\inf_{\substack{ a,b\in U \\ a\neq b}}\frac{1}{\left\vert
1-a\right\vert \left\vert 1-b\right\vert }=1.
\end{equation*}

The function $f:U\rightarrow \mathbb{C}$ defined by $f\left( z\right) =\frac{%
z^{2}}{1-z}$ is analytic in $U$ and satisfies%
\begin{equation*}
\left\vert f^{\prime }\left( z\right) -g^{\prime }\left( z\right)
\right\vert =1\leq K\left( g,U\right) ,\qquad z\in U,
\end{equation*}%
and therefore by Theorem \ref{main theorem 2} it follows that $f$ is
univalent in the unit disk.

The univalence of $f$ does not follow however by the
Noshiro-Warschawski-Wolff univalence criterion since
$\func{Re}f^{\prime }\left( z\right) $ takes (arbitrarily small)
negative values for $z\in U$ sufficiently close to $1$.
\end{example}

As another application of Theorem \ref{main theorem 2}, in the next result
we show that by perturbing the coefficients of the Taylor series of an
univalent function, the resulting function is also univalent. More
precisely, we have the following:

\begin{theorem}
\label{Application to Taylor series} Let $g:U\rightarrow \mathbb{C}$
be an analytic univalent function with
Taylor series representation%
\begin{equation}
g\left( z\right) =\sum_{n=0}^{\infty }b_{n}z^{n},\qquad z\in U\text{.}
\label{Taylor series for g}
\end{equation}

If the coefficients $a_{0},a_{1},\ldots \in \mathbb{C}$ satisfy the
inequality
\begin{equation}
\sum_{n=1}^{\infty }n\left\vert a_{n}-b_{n}\right\vert <K\left( g,U\right)
\label{hypothesis on coefficients of Taylor series}
\end{equation}%
then the function $f:U\rightarrow \mathbb{C}$ defined by
\begin{equation}
f\left( z\right) =\sum_{n=0}^{\infty }a_{n}z^{n},\qquad z\in U,
\label{Taylor series for f}
\end{equation}%
is analytic and univalent in $U$.
\end{theorem}

\begin{proof}
Since $g$ is univalent in $U$, the radius of convergence of the Taylor
series (\ref{Taylor series for g}) is at least $1$, hence
\begin{equation*}
\lim \sup \sqrt[n]{\left\vert b_{n}\right\vert }\leq 1,
\end{equation*}%
and therefore $\left\vert b_{n}\right\vert \leq 1$ for all $n$ sufficiently
large.

Using the hypothesis (\ref{hypothesis on coefficients of Taylor series}) we
obtain%
\begin{equation*}
\lim \sup \sqrt[n]{\left\vert a_{n}\right\vert }\leq \lim \sup \sqrt[n]{%
\left\vert b_{n}\right\vert +\left\vert a_{n}-b_{n}\right\vert }\leq \lim
\sup \sqrt[n]{1+\frac{K\left( g,U\right) }{n}}=1,
\end{equation*}%
and therefore the radius of convergence of the series in (\ref{Taylor series
for f}) is at least $1$, thus the function $f$ is well defined by (\ref%
{Taylor series for f}) and it is analytic in $U$.

Since
\begin{eqnarray*}
\left\vert f^{\prime }\left( z\right) -g^{\prime }\left( z\right)
\right\vert &=&\left\vert \sum_{n=0}^{\infty
}na_{n}z^{n-1}-\sum_{n=0}^{\infty }nb_{n}z^{n-1}\right\vert \\
&\leq &\sum_{n=1}^{\infty }n\left\vert a_{n}-b_{n}\right\vert \left\vert
z\right\vert ^{n-1} \\
&\leq &\sum_{n=1}^{\infty }n\left\vert a_{n}-b_{n}\right\vert \\
&<&K\left( g,U\right) ,
\end{eqnarray*}%
for any $z\in U$, by Theorem \ref{main theorem 2} follows that $f$ is
univalent in $U$, concluding the proof.
\end{proof}

Using a comparison with the generalized harmonic series, from the
above we can obtain the following:

\begin{corollary}
Let $g:U\rightarrow \mathbb{C}$ be an analytic univalent function with
Taylor series representation%
\begin{equation}
g\left( z\right) =\sum_{n=0}^{\infty }b_{n}z^{n},\qquad z\in U\text{.}
\end{equation}

If the coefficients $a_{0},a_{1},\ldots \in \mathbb{C}$ satisfy the
inequality
\begin{equation}
\left\vert a_{n}-b_{n}\right\vert <K\left( g,U\right) \frac{\zeta \left(
p\right) }{n^{p+1}},\qquad n=1,2,\ldots ,
\end{equation}%
for some $p>1$ ($\zeta $ denotes the Riemann zeta function), then the
function $f:U\rightarrow \mathbb{C}$ defined by
\begin{equation}
f\left( z\right) =\sum_{n=0}^{\infty }a_{n}z^{n},\qquad z\in U,
\end{equation}%
is analytic and univalent in $U$.
\end{corollary}

\begin{example}
Considering the function $g\left( z\right) =\frac{z}{1-z}=\sum_{n=1}^{\infty
}z^{n}$ defined in Example \ref{Exemplul 2}, which is analytic and univalent
in $U$ and has $K\left( g,U\right) =1$, from the previous theorem it follows
that the function $f:U\rightarrow \mathbb{C}$ defined by $f\left( z\right)
=\sum_{n=0}^{\infty }a_{n}z^{n}$ is analytic and univalent in $U$ if the
coefficients $a_{n}$ satisfy the inequality%
\begin{equation*}
\sum_{n=1}^{\infty }n\left\vert a_{n}-1\right\vert <1.
\end{equation*}

Using for example the fact that $\zeta \left( 2\right) =\frac{\pi ^{2}}{6}%
\approx 1.645$, from the previous corollary it follows that the function $f$
is also analytic and univalent in $U$ if the coefficients $a_{n}$ satisfy
the inequality%
\begin{equation*}
\left\vert a_{n}-1\right\vert \leq \frac{\pi ^{2}}{6n^{3}}\approx \frac{1.645%
}{n^{3}},\qquad n=1,2,\ldots
\end{equation*}
\end{example}

\end{document}